\documentclass[a4paper,11pt,oneside]{article}
\usepackage{amsmath,amsthm,amsfonts,amssymb,ifpdf}
\usepackage{amssymb}
\usepackage{amsmath}
\usepackage{amsthm}
\usepackage{graphicx}
\usepackage[all]{xy}
\usepackage{enumerate}
\usepackage{tikz-cd}
\usepackage{subcaption}
\usepackage{authblk}

\ifpdf        
\usepackage[
pdftex,
colorlinks,%
linkcolor=blue,citecolor=red,urlcolor=blue,
hyperindex,%
plainpages=false,%
bookmarksopen,%
bookmarksnumbered%
]{hyperref} 
\usepackage{thumbpdf}
\else
\usepackage{hyperref}
\fi

\usepackage{theoremref} 

\usepackage{tikz}
\usetikzlibrary{positioning, arrows.meta}

\usepackage{amsmath, amscd}
\usepackage{amssymb, amsthm}
\usepackage{mathrsfs, calrsfs}
\usepackage{graphicx}

\newtheorem{theorem}{Theorem}[section]
\newtheorem{remark}{Remark}[section]

\newtheorem{lemma}{Lemma}[section]
\newtheorem{example}{Example}[section]

\title{Decay estimate in a viscoelastic plate equation with past history, nonlinear damping, and logarithmic nonlinearity}

\author{Bhargav Kumar Kakumani\thanks{bhargav@hyderabad.bits-pilani.ac.in}}
\author{Suman Prabha Yadav}
\affil{Department of Mathematics, BITS-Pilani, Hyderabad Campus, Hyderabad, India.}

\date{\today}
\begin{document}
\maketitle

\begin{abstract}
	In this article, we consider a viscoelastic plate equation with past history, nonlinear damping, and logarithmic nonlinearity. We prove explicit and general decay rate results of the solution to the viscoelastic plate equation with past history. Convex properties, logarithmic inequalities, and generalised Young's inequality are mainly used to prove the decay estimate.
\end{abstract}
\section{Introduction}
In this article, we consider the decay rate results of the solution to a viscoelastic plate equation with past history. Let $\Omega$ be a smooth bounded domain of $\mathbb{R}^n$ and $u$ denotes the transverse displacement of waves. Assume $u_0,u_1$ are given initial data, then the partial differential equation is governed by the plate equation is given by:
\begin{equation}
\label{eqn_s1}
\left\{ \begin{array}{ll}
\left |u_{t} \right |^{\rho }u_{tt}+\Delta ^{2}u+\Delta ^{2}u_{tt}+u-\displaystyle\int_{0}^{\infty}b(s)\Delta ^{2}u(t-s)ds +h(u_{t})=ku \ln \left | u \right |,\ \text{in} \  \Omega\times(0,\infty), \medskip \\
u\left ( x,t \right )=\frac{\partial u}{\partial \nu}\left ( x,t \right )=0, \hspace{0.2cm}  \text{in} \hspace{0.2cm}  \, \partial \Omega \times(0,\infty),  \medskip\\
u(x,-t)=u_{0}(x,t), \hspace{0.5cm}  u_{t}(x,0)=u_{1}(x), \hspace{0.2cm} \textrm{in} \hspace{0.2cm}\, \Omega, t\geq 0,
\end {array} \right.
\end{equation}
where $\nu$ is the outer unit normal to $\partial\Omega$, $b$, $h$ are functions (defined later) and $\rho $ is a positive constant ($\rho>0$ if $n\geq 2$ and $0 < \rho \leq \frac{2}{n-2}$ if $n \geq 3 $).\\
\noindent Viscoelasticity takes into account of time independent solid behaviour (namely, elastic) and time dependent fluid behaviour (namely, viscosity). Some properties of viscoelastic materials are similar to those of elastic solids and some with Newtonian viscous fluids. Because of the significant advancements in the rubber and plastics industries, the importance of material viscoelastic characteristics has been recognised.\\
We begin our review with Dafermos' pioneer paper \cite{C.M._1970}, in which the author presented the following one-dimensional viscoelastic problem:
\begin{equation*}
\left\{ \begin{array}{ll}
\rho u_{tt}=cu_{xx}-\int_{-\infty }^{t}g(t-\tau )u_{xx}\mathrm{d} \tau, \hspace{.2cm} x \in [0,1],\\
u(0,t)=u(1,t)=0, \hspace{.2cm} t \in (-\infty, \infty).
\end {array} \right.
\end{equation*}
Dafermos' proved that for smooth monotonic decreasing relaxation functions, the solutions go to zero as $t$ tends to infinity, he also established an existence result. However, no rate of decay has been specified. After that many author tries to prove the local/global existence results and also tries to calculate the explicit decay result. 
In \cite{Guesmia2020} the authors considered viscoelastic wave equation with infinte memory given by the equation:
\begin{equation*}
u_{tt}- \Delta u+\int_{0}^{\infty}g(s)\Delta u(t-s)\mathrm{d}s=0, \hspace{0.2cm} \mathrm{in} \hspace{0.2cm} \Omega \times \mathbb{R}^+,
\end{equation*} 
and established energy decay result without taking any assumption on the boundedness of initial data and any growth constraint on the damping term. Aissa Guesmia and  S.A.Messaoudi in \cite{Aissa_salim} considered:
\begin{equation*}
u_{tt}-\Delta u +\int_{o}^{t}g_{1}(t-s) \mathrm{div}(a_{1}(x) \nabla u(s)) \mathrm{d}s
+\int_{0}^{\infty}g_{2}(s)\mathrm{div}(a_{2}(x)\nabla u (t-s)) \mathrm{d}s=0,\ \mathrm{in}\ \Omega \times \mathbb{R}^+,\\
\end{equation*}
where $g_{1}$ and $g_{2}$ are two positive nonincreasing functions defined on $\mathbb{R}^+$, and $a_{1},a_{2}$ are nonnegative bounded function defined on $\Omega$. The authors in \cite{Aissa_salim} proved the general decay result. Later, A.M. Mahdi in \cite{adel_M} considered the problem (\ref{eqn_s1}) with $h=0$ and established an explicit and general decay rate results. For more work related to the past history problems, refer to \cite{GUESMIA2011748, MM_2020, MA_TF, Pata_VCG, KANG2015509, AL_MM} (and the references there in). Throughout this paper, we consider the following hypothesis:\\
\begin{enumerate}
	\label{eq_s2}
	\item[(H1)] Let $b:\mathbb{R}^{+}\rightarrow \mathbb{R}^{+}$
	be a nonincreasing $C^1$- function which satisfies:
	\begin{equation*}
	0<b(0),\, 1-\int_{0}^{\infty}b(\tau)d\tau=l>0.
	\end{equation*}
	\label{eqn_s3}
	\item[(H2)] Assume that $B:(0,\infty )\rightarrow (0,\infty )$ be a $C^1$ function which is linear or strictly convex $C^2$ function and strictly increasing on $(0,r_1],\ {\rm where}\ r_1\leq b(0)$, $B(0)=B'(0)=0 $ and $B$ satisfies 
	\begin{equation*}
	b'(t)\leq -\xi(t)B(b(t)), \quad \forall t\geq 0,
	\end{equation*}
	where $\xi:\mathbb{R}^{+}\rightarrow \mathbb{R}^{+}$ is a $C^1$ nonincreasing positive function with $\xi(0)>0$.
	\item[(H3)] Let $h :\mathbb{R} \rightarrow \mathbb{R}$ is a nondecreasing continuous function which satisfies (for some $c_1,c_2,\epsilon$ are positive constants):
	\begin{center}
		$\tilde{h}(\left | t \right |) \leq \left | h(t) \right | \leq \tilde{h}^{-1}(\left |t \right | )$,\ $\forall$  $\left | t \right | \leq  \epsilon,$ \\
		$c_{1}\left | t \right |\leq \left | h(t) \right |  \leq c_{2} \left | t \right |$,\ $\forall$ $\left | t \right | \geq  \epsilon$,
	\end{center}
	where $\tilde{h}\in C^{1}(\mathbb{R}^{+})$ with $ \tilde{h}(0)=0 $ which is strictly increasing function. When $ \tilde{h} $ is nonlinear, define $H$ to be a strictly convex $ C^{2} $ function in $ (0,r_2],$ where $ r_2 > 0 $ such that $ H(t) = \sqrt{t}\tilde{h}(\sqrt{t}) $.
\end{enumerate}
\begin{remark}
	\label{Remark_1}
	Since $B$ is strictly convex on $(0,r_{1}]$ and $B(0)=0$, then
	\label{con_B}
	\begin{equation*}
	B(\theta t)\leq \theta B(t), \hspace{.2cm} 0\leq \theta \leq 1 \hspace{.2cm} \text{and} \hspace{.2cm}t \in (0,r_{1}].
	\end{equation*}
\end{remark}
In this article, we give certain notations and declare existence results in Section 2. In addition, we express a couple of Lemmas in Section 2 that will be useful later. We state and establish the decay rate estimate in Section 3, as well as present an example to demonstrate the decay rate.

\section{Preliminaries and Existence results}
In this section, we review Dafermos' theory (see \cite{C.M._1970}) and define the energy functional that is relevant to our problem. We also state a local existence result and a couple of additional Lemma's that will be useful later on. We introduce the function $\eta$ as follows:
\begin{equation}
\label{R_1}	
\eta ^{t}(x,s)=u(x,t)-u(x,t-s), \hspace{.2cm}  \forall s, t \geq 0,\ x\in \Omega, 
\end{equation}
then  the initial and boundary conditions are obtained as follows:
\begin{equation*}
\left\{ \begin{array}{ll}
\eta ^{t}(x,0)=0, \hspace{.2cm}\forall t\geq 0,\ x \in \Omega, \\
\eta ^{t}(x,s)=0, \hspace{.2cm} \forall s,t \geq 0,\ x \in \partial\Omega, \\
\eta ^{0}(x,s)= \eta_{0}(x,s)=u_{0}(x,0)-u_{0}(x,s), \hspace{.2cm}\forall  s\geq 0,\ x \in \Omega.
\end{array} \right.
\end{equation*}
Observe that, (\ref{R_1}) implies
\begin{equation}
\eta_{s}^{t}(x,s)+\eta_{t}^{t}(x,s)=u_{t}(x,t).
\end{equation}
After combining (\ref{eqn_s1}) and (\ref{R_1}), we obtain the following system
\begin{equation}
\label{sys_1}
\left\{ \begin{array}{ll}
\left |u_{t} \right |^{\rho }u_{tt}+l\Delta ^{2}u+\Delta ^{2}u_{tt}+u-\displaystyle\int_{0}^{\infty}b(s)\Delta ^{2}\eta^{t}ds +h(u_{t})=ku \ln \left | u \right |, \quad  t\geq 0,\  x \in \Omega,\medskip \\	
\eta _{s}^{t}(x,s)+\eta_{t}^{t}(x,s)-u_{t}(x,t)=0, \hspace{.2cm} \forall s,t\geq 0, \hspace{.2cm} x\in \Omega, 
\end {array} \right.
\end{equation}
with the following initial and boundary data
\begin{equation}
\left\{ \begin{array}{ll}
u(x,-t)=u_{0}(x,t), \hspace{.2cm} u_{t}(x,0)=u_{1}(x), \hspace{.2cm} \forall t\geq 0, \forall x\in \Omega, \\
\eta^{0}(x,s)=\eta_{0}(x,s)=u_{0}(x,0)-u_{0}(x,s), \hspace{.2cm} \eta^{t}(x,0)=0,\hspace{.2cm} \forall s,t\geq 0, \hspace{.2cm} \forall x \in \Omega,\\
u(x,t)=0 \hspace{.2cm} \eta^{t}(x,s)=0 \hspace{.2cm} {\rm in}\ \hspace{.2cm} \partial  \Omega, \hspace{.2cm} \forall s,t\geq 0.
\end {array} \right.
\end{equation}
The energy functional associated with system (\ref{sys_1}) is given by
\begin{equation}
\begin{array}{ll}
\label{E_fn}
E(t) =\frac{1}{\rho +2}\left \| u_{t} \right \|_{\rho +2}^{\rho +2}+\frac{l}{2}\left \| \Delta u \right \|_{2}^{2}+\frac{1}{2}\left \| \Delta u_{t} \right \|_{2}^{2} -\frac{k}{2}\int_{\Omega }u^{2} \ln \left | u \right |\mathrm{d}x+ \frac{k+2}{4}\left \| u \right \|_{2}^{2}+\frac{1}{2}(b\circ\Delta \eta^{t}),
\end{array}
\end{equation}
where $(b\circ \Delta n^{t})(t)=\int_{0}^{\infty}b(s)\left | \Delta \eta^{t} \right |^2 \mathit{d} s$ and $\|..\|_2=\|..\|_{L^2(\Omega)}$.  Differentiating  $E(t)$ with $t$ and making use of system (\ref{sys_1}), we obtain
\begin{equation}
\label{E'_t}
E'(t) \leq \frac{1}{2}(b\circ\Delta \eta^{t})-\int_{\Omega} u_{t}h(u_{t}) \leq 0.
\end{equation}
\begin{remark}
	\label{b_1}
	For any $\epsilon_{0} \in (0,1), $ we obtain that
	\begin{center}
		$(b \circ \Delta \eta^{t})= (b \circ \Delta \eta^{t})^{\frac{\epsilon_{0}}{1+\epsilon_{0}}} (b \circ \Delta \eta^{t})^{\frac{1}{1+\epsilon_{0}}} \leq c(b \circ \Delta \eta^{t})^{\frac{1}{1+\epsilon_{0}}}. $
	\end{center}
\end{remark}
\begin{theorem}
	\label{local_existence}
	Let $(u_{0}(.,0),u_{1}) \in H_{0}^{2}(\Omega) \times H_{0}^{2}(\Omega),$ and assume that the hypothesis (H1)-(H3) hold.  Then the problem \eqref{eqn_s1} has weak solution on $[0,T].$
\end{theorem}
\noindent The proof of the above theorem can be obtained by following similar lines given in \cite{bharu_suman}.
\begin{lemma} (Cf. \cite{adel_M}, Lemma $3.1$)
	There exists a constant $M > 0$ such that 
	\begin{equation*}
	\int_{t}^{\infty}b(s)(\Delta \eta^{t}(x,s))^{2} \mathit{d}s\mathit{d}x \leq M f_{1}(t),
	\end{equation*}
	where $f_{1}(t)= \int_{0}^{\infty}b(s+t)(1+\left \| \Delta u_{0}(s) \right \|_2^{2})\mathit{d}s.$
\end{lemma}
\begin{lemma}
	\label{LE_1}
	Assume that (H1)-(H3) holds, for some $\epsilon_{0} \in (0,1)$ and $0<E(0)<d\footnote{where $d$ denotes the depth of the potential well, see \cite{bharu_suman} for such construction of $d.$}$. Define the functional's
	\begin{equation*}
	\psi _{1}(t):=\frac{1}{\rho+1}\int_{\Omega}\left | u_{t} \right |^{\rho}u_{t} u\mathit dx+\int_{\Omega}\Delta u \Delta u_{t} \mathit d x,
	\end{equation*}
	\begin{equation*}
	\psi_{2}(t):=-\int _{\Omega}\Big(\Delta^{2}u_{t}+\frac{1}{\rho+1}\left | u_{t} \right |^{\rho}u_{t}\Big)\int_{0}^{\infty}b(s)\eta^{t}(s) \mathit ds\mathit dx,
	\end{equation*}
	and $L(t):=mE(t)+\epsilon \psi_{1}(t)+\psi_{2}(t),$  where $m, \epsilon \geq 0$ then $L$ satisfies the following:\\
	(I)  $L \sim E$ ($i.e.,\ \alpha_1E(t)\leq L(t) \leq \alpha_2E(t),$ for some $\alpha_1,\alpha_2>0$), \\
	(II) $L'(t) \leq -mE(t)+c(b \circ \Delta \eta^{t})(t)+c(b \circ \Delta \eta^{t})^\frac{1}{1+\epsilon_{0}}(t)+c \int_{\Omega } h^2(u_{t}) \mathit{dx}.$
	
\end{lemma}
\begin{lemma}  (Cf. \cite{gharabli_etal_2019}, Lemma $4.1$)
	Let $h$ satisfies $\mathrm{(H3)}$. Then, the solution of (\ref{eqn_s1}) satisfies
	\begin{equation*}
	\int_{\Omega} h^2(u_{t}) \mathit{dx} \leq c \Big( H^{-1}(G(t))-E'(t)\Big), \hspace{0.2cm}
	\end{equation*}
	where $G(t):=\frac{1}{\left | \Omega_{1} \right |}\int_{\Omega} u_{t}h(u_{t})\mathit{dx} \leq -cE'(t)$ and  ${\Omega_{1}=\{x \in \Omega: \left | u_{t} \right | \leq \epsilon}\}$ for some $ c,\epsilon >0. $
\end{lemma}
\begin{lemma} (Cf. \cite{MM_2020}, Lemma $3.3$)
	\label{int_b}
	Let $\delta_{0} > 0$  and assume that (H1)-(H2) holds, then we have $\forall t\geq 0,$
	\begin{equation*}
	\int_{0}^{t}b(s)\left | \Delta \eta^{t} \right |^2\mathit ds \leq \Big(\frac{1+t}{\delta _{0}}\Big )B^{-1}\Big(\frac{\delta _{0} \mu  (1+t)}{t\xi (t)}\Big ), 
	\end{equation*}
	where 
	\begin{equation*}
	\label{mu}
	\mu (t):=\int_{0}^{t}b'(s)\left | \Delta \eta^{t} \right |^2 \mathit ds \leq -cE'(t).
	\end{equation*}
\end{lemma}
\section{Decay result}
We state and prove the major result in this section. We also provide an example to demonstrate the decay rate result. We introduce a few notations and a function for this purpose:
\begin{equation}
\label{w_23}
W_{1}(t):=\int_{t}^{1}\frac{1}{s W'(s)}\mathit{ds}, \quad W_{2}(t)=t W'(t), \quad W_{3}(t)=t(W')^{-1}(t), 
\end{equation}
where $W(t)=((B^{-1}(t))^\frac{1}{1+\epsilon_{0}}+H^{-1})^{-1}.$ Further, denote $\mathcal{S}$ to be the class of functions $\chi:(0, \infty) \rightarrow [0, \infty)$ be a $C^1$ function satisfying $\chi \leq 1\ {\rm and} \  \chi' \leq 0$. Also, for fixed $C_{1},C_{2} >0,$ assume that $\chi$ satisfies the estimate:
\begin{equation}
\label{w_3}
C_{2}W_{3}^{*}\Big [\frac{c}{\delta_{1}}q(t)f_{1}^{\frac{1}{1+\epsilon_{0}}}(t)\Big ] \leq C_{1}\Big ( W_{2}\Big (\frac{W_{4}(t)}{\chi(t)}\Big )-\frac{W_{2}(W_{4}(t))}{\chi(t)}\Big ),
\end{equation}
where $W_{3}^{*}$ (defined in \eqref{young_def}) is convex conjugate of $W_{3}$. Let $\delta_{1}, c>0$ are generic constants, $q(t)=(t+1)^{-\frac{1}{1+\epsilon_{0}}}$  and 
\begin{equation}
\label{w_4}
W_{4}(t)=W_{1}^{-1}\Big(C_{1}\int_{0}^{t} \xi(s) \mathit{ds}\Big ).
\end{equation}
\textbf{Note:} For $\epsilon$ small enough with $0<\epsilon \leq 1,\ \epsilon W_{4}(s)\in \mathcal{S}.$ So, the set $\mathcal{S}$ is non empty.
\begin{theorem}
	Under the hypothesis of Lemma \ref{LE_1}, there exists a constant $C>0$ such that the solution to problem  (\ref{eqn_s1}) satisfies, \\
	\begin{equation}
	\label{energy_est}
	E(t) \leq \frac{C W_{4}(t)}{\chi(t)q(t)}, \hspace{0.2cm} \forall t \geq 0.
	\end{equation}
\end{theorem}
\begin{proof}
	Using Lemmas (\ref{b_1})- (\ref{int_b}), note that
	\begin{equation*}
	\begin{array}{ll}
	L'(t) \leq -mE(t)+\Big[\Big(\frac{t+1}{\delta _{0}}\Big )B^{-1}\Big(\frac{\delta _{0} \mu  (t)}{(t+1)\xi (t)}\Big )\Big]^{\frac{1}{1+\epsilon_{0}}} \medskip \\
	\hspace{1.5cm}+\Big[c\int_{0}^{\infty}b(t+s)(1+\left \| \Delta u_{0}(s) \right \|_2^{2})\mathit{d}s\Big]^\frac{1}{1+\epsilon_{0}}+cH^{-1}(G(t))-cE'(t).
	\end{array}
	\end{equation*}\\
	Assume $L_{1}(t):=(L+cE)(t)$, then the above inequality can be written as
	\begin{equation*}
	\begin{array}{ll}
	L_{1}'(t) \leq -mE(t)+c\Big[\Big(\frac{t+1}{\delta _{0}}\Big )B^{-1}\Big(\frac{\delta _{0} \mu  (t)}{(t+1)\xi (t)}\Big )\Big]^{\frac{1}{1+\epsilon_{0}}} +\big[cf_1(t)\big]^\frac{1}{1+\epsilon_{0}}+cH^{-1}(G(t)).
	\end{array}
	\end{equation*}
	Using Remark (\ref{Remark_1}) with $\theta=\frac{1}{t+1}$ for all $t>0,$ we obtain\\
	\begin{equation*}
	B^{-1}\Big(\frac{\delta _{0} \mu  (t)}{(t+1)\xi (t)}\Big )\leq B^{-1}\Big(\frac{\delta _{0}q(t) \mu  (t)}{\xi (t)}\Big ).
	\end{equation*}
	Hence,
	\begin{equation}
	\begin{array}{ll}
	\label{ss_1}
	L_{1}'(t) \leq -mE(t)+c\Big(\frac{t+1}{\delta _{0}}\Big )^{\frac{1}{1+\epsilon_{0}}}B^{-1}\Big(\frac{\delta _{0}q(t) \mu  (t)}{\xi (t)}\Big )^{\frac{1}{1+\epsilon_{0}}} +\big[cf_1(t)\big]^\frac{1}{1+\epsilon_{0}}+\frac{c}{\delta _0 q(t)} H^{-1}(G(t)q(t)).
	\end{array}
	\end{equation}
	Denote\quad	$\beta(t):=\mathrm{max}\Big (\frac{\delta _{0}q(t) \mu  (t)}{\xi (t)}, G(t)q(t) \Big )$ and recall  $W(t)=((B^{-1})^{\frac{1}{1+\epsilon_{0}}}+H^{-1})^{-1}(t).$\\
	Then (\ref{ss_1}) becomes, for any $t\geq 0$ and $\epsilon_{0} \in (0, 1),$
	\begin{equation*}
	\begin{array}{ll}
	\label{ss_2}
	L_{1}'(t) \leq -mE(t)+c\frac{1}{\delta _{0}q(t)}W^{-1}(\beta(t))+cf_{1}^{\frac{1}{1+\epsilon_{0}}}.
	\end{array}
	\end{equation*}
	Let $0<\epsilon_{1} <r:=\min\{r_1,r_2\}$, define the functional $L_{2}$ as
	\begin{equation*}
	\label{l2}
	L_{2}(t):=W'\Big( {\epsilon_{1}q(t)}\frac{E(t)}{E(0)}\Big ) L_{1}(t),
	\end{equation*}
	then it is easy to observe that $L_{2} \sim E$, also
	\begin{equation}
	\begin{array}{ll}
	\label{ss_3}
	L_{2}'(t) \leq -mE(t)W'\Big( {\epsilon_{1}q(t)}\frac{E(t)}{E(0)}\Big )+c\frac{1}{\delta _{0}q(t)}W'\Big( {\epsilon_{1}q(t)}\frac{E(t)}{E(0)}\Big )W^{-1}(\beta(t)) +cf_{1}^{\frac{1}{1+\epsilon_{0}}}W'\Big( {\epsilon_{1}q(t)}\frac{E(t)}{E(0)}\Big ),
	\end{array}
	\end{equation}
	Let the convex conjugate of $W$ is denoted by $W^{*}$ and is defined as
	\begin{equation}
	\label{young_def}
	W^{*}(\tau)=\tau(W')^{-1}(\tau)-W[(W')^{-1}(\tau)],\quad \tau \in (0, W'(r)].
	\end{equation}
	Using Generalized Young's inequality, $W^*$ satisfies the following estimate
	\begin{equation}
	\label{young}
	\tilde{A}\tilde{B} \leq W^{*}(\tilde{A})+W(\tilde{B}), \quad \tilde{A} \in (0,W'(r)],\ {\rm and}\  \tilde{B} \in (0,r ].
	\end{equation}
	Therefore with $\tilde{A}=W'\Big( {\epsilon_{1}q(t)}\frac{E(t)}{E(0)}\Big )$ and $\tilde{B}=W^{-1}(\beta(t))$, \eqref{ss_3} leads to
	\begin{equation*}
	\begin{array}{ll}
	\label{ss_4}
	L_{2}'(t) \leq -mE(t)W'\Big( {\epsilon_{1}q(t)}\frac{E(t)}{E(0)}\Big ) +\frac{c}{\delta _{0}q(t)}W^{*}\Big(W'\Big( {\epsilon_{1}q(t)}\frac{E(t)}{E(0)}\Big )\Big ) +c\frac{\beta(t)}{\delta_{0}q(t)}+cf_{1}^{\frac{1}{1+\epsilon_{0}}}W'\Big( {\epsilon_{1}q(t)}\frac{E(t)}{E(0)}\Big ).
	\end{array}
	\end{equation*}
	Multiplying the above equation by $\xi (t)$, we get
	\begin{equation*}
	\begin{array}{ll}
	\label{ss_5}
	\xi (t)L_{2}'(t) &\leq -m \xi (t)E(t)W'\Big( {\epsilon_{1}q(t)}\frac{E(t)}{E(0)}\Big ) +\frac{c \xi (t)}{\delta _{0}q(t)}W^{*}\Big(W'\Big( {\epsilon_{1}q(t)}\frac{E(t)}{E(0)}\Big )\Big )\medskip \\
	&\hspace{0.5cm} +c \xi (t)\frac{\beta(t)}{\delta_{0}q(t)}+c \xi(t)f_{1}^{\frac{1}{1+\epsilon_{0}}}W'\Big( {\epsilon_{1}q(t)}\frac{E(t)}{E(0)}\Big ).
	\end{array}
	\end{equation*}
	Define a functional $L_{3}:=\xi L_{2}+cE \sim E$. Since $\xi (t) \beta (t) \leq -cE'(t)$ and  $W^{*}(W'(t)) \leq tW'(t)$,  we obtain
	\begin{equation*}
	\begin{array}{ll}
	\label{ss_6}
	L_{3}'(t) & \leq -m E(t)\xi (t)W'\Big( {\epsilon_{1}q(t)}\frac{E(t)}{E(0)}\Big ) + \frac{c}{\delta _{0}}\epsilon_{1}\xi (t)\frac{E(t)}{E(0)}W'\Big( {\epsilon_{1}q(t)}\frac{E(t)}{E(0)}\Big ) +c \xi(t)f_{1}^{\frac{1}{1+\epsilon_{0}}}W'\Big( {\epsilon_{1}q(t)}\frac{E(t)}{E(0)}\Big ) \medskip \\
	& \leq -\Big (\frac{mE(0)}{\epsilon_{1}}-\frac{c}{\delta_{0}}\Big) \epsilon_{1} \xi (t) \frac{E(t)}{E(0)}W'\Big( {\epsilon_{1}q(t)}\frac{E(t)}{E(0)}\Big ) +c \xi(t)f_{1}^{\frac{1}{1+\epsilon_{0}}}W'\Big( {\epsilon_{1}q(t)}\frac{E(t)}{E(0)}\Big ).
	\end{array}
	\end{equation*}
	Consequently, from \eqref{w_23} and choosing $\epsilon_{1}$ such that $k:=\Big(\frac{mE(0)}{\epsilon_{1}}-c \Big)>0,$ we obtain
	\begin{equation}
	\begin{array}{ll}
	\label{ss_40}
	L_{3}'(t) & \leq -k \epsilon_{1}\xi (t) \frac{E(t)}{E(0)}W'\Big( {\epsilon_{1}q(t)}\frac{E(t)}{E(0)}\Big ) +c \xi(t)f_{1}^{\frac{1}{1+\epsilon_{0}}}W'\Big( {\epsilon_{1}q(t)}\frac{E(t)}{E(0)}\Big )\medskip \\
	&\leq -k  \frac{\xi(t)}{q(t)}W_{2}\Big( {\epsilon_{1}q(t)}\frac{E(t)}{E(0)}\Big )+c \xi(t)f_{1}^{\frac{1}{1+\epsilon_{0}}}W'\Big( {\epsilon_{1}q(t)}\frac{E(t)}{E(0)}\Big ).
	\end{array}
	\end{equation}
	Since $W'_{2}(t)=W'(t)+tW''(t)$, using the property of  $W$, we conclude that $W_{2}(t), W'_{2}(t)>0$ on $(0,r]$. Making use of (\ref{young}) with $\tilde{A}=W'\Big( {\epsilon_{1}q(t)}\frac{E(t)}{E(0)}\Big )$ and $\tilde{B}=\Big [\frac{c}{\delta_{1}}f^{\frac{1}{1+\epsilon_{0}}}_{1}\Big]$ where $\delta_{1}>0$, we get
	\begin{equation}
	\begin{array}{ll}
	\label{ss_8}
	cf^{\frac{1}{1+\epsilon_{0}}}_{1}W'\Big( {\epsilon_{1}q(t)}\frac{E(t)}{E(0)}\Big )&=\frac{\delta_{1}}{q(t)}\Big[\frac{c}{\delta_{1}}q(t)f_{1}^{\frac{1}{1+\epsilon_{0}}}\Big]W'\Big( {\epsilon_{1}q(t)}\frac{E(t)}{E(0)}\Big ) \medskip \\
	&\leq \frac{\delta_{1}}{q(t)}W_{3}\Big(W'\Big( {\epsilon_{1}q(t)}\frac{E(t)}{E(0)}\Big )\Big)+\frac{\delta_{1}}{q(t)} W^{*}\Big[\frac{c}{\delta_{1}}q(t)f_{1}^{\frac{1}{1+\epsilon_{0}}}\Big] \medskip \\
	&\leq \frac{\delta_{1}}{q(t)}\Big( {\epsilon_{1}q(t)}\frac{E(t)}{E(0)}\Big )W'\Big( {\epsilon_{1}q(t)}\frac{E(t)}{E(0)}\Big )+ \frac{\delta_{1}}{q(t)}W_{3}^{*}\Big[\frac{c}{\delta_{1}}q(t)f_{1}^{\frac{1}{1+\epsilon_{0}}}\Big] \medskip \\
	&\leq \frac{\delta_{1}}{q(t)} W_{2}\Big( {\epsilon_{1}q(t)}\frac{E(t)}{E(0)}\Big )+\frac{\delta_{1}}{q(t)}W_{3}^{*}\Big[\frac{c}{\delta_{1}}q(t)f_{1}^{\frac{1}{1+\epsilon_{0}}}\Big].
	\end{array}
	\end{equation}
	Now, combining (\ref{ss_40}) and (\ref{ss_8}), choose $\delta_{1}$ small enough so that $k_{1}=(k-\delta_{1})>0$, we have 
	\begin{equation}
	\begin{array}{ll}
	\label{ss_9}
	L'_{3}(t) \leq-k_{1}\frac{\xi (t)}{q(t)}W_{2}\Big( {\epsilon_{1}q(t)}\frac{E(t)}{E(0)}\Big )+\frac{\delta_{1}\xi (t)}{q(t)}W_{3}^{*}\Big[\frac{c}{\delta_{1}}q(t)f_{1}^{\frac{1}{1+\epsilon_{0}}}\Big],
	\end{array}
	\end{equation}
	using nonincreasing property of $W_{2}$ and for some $\gamma_{1}$, $\gamma_{2}>0$ which satisfies
	$\gamma_1L_{3}(t) \leq E(t) \leq \gamma_{2}L_{3}(t).$ We have, for some $\alpha=\frac{\gamma_{1}}{E(0)}>0,$
	\begin{equation*}
	\begin{array}{ll}
	W_{2}\Big( {\epsilon_{1}q(t)}\frac{E(t)}{E(0)}\Big ) \geq W_{2}\alpha L_{3}(t)q(t).	
	\end{array}
	\end{equation*}
	Assume $L_{4}(t):=\alpha L_{3}(t)q(t)$, then using \eqref{ss_9} we have 
	\begin{equation}
	\begin{array}{ll}
	\label{ss_10}
	L'_{4}(t) & \leq\alpha q(t)\Big(-k_{1}\frac{\xi (t)}{q(t)}W_{2}\Big( {\epsilon_{1}q(t)}\frac{E(t)}{E(0)}\Big )+\frac{\delta_{1}\xi (t)}{q(t)}W_{3}^{*}\Big[\frac{c}{\delta_{1}}q(t)f_{1}^{\frac{1}{1+\epsilon_{0}}}\Big]\Big) \medskip \\
	& \leq  -C_{1}\xi (t)W_{2}(L_{4}(t))+C_{2}\xi (t) W_{3}^{*}\Big[\frac{c}{\delta_{1}}q(t)f_{1}^{\frac{1}{1+\epsilon_{0}}}\Big ],
	\end{array}
	\end{equation}
	where $C_{1}=\alpha k_{1}>0$ and $C_{2}=\alpha \delta_{1}>0$. Since, $L_{3}\sim E,$ we have $L_{4}(t) \leq \alpha_{0}E(t)q(t)$ for some $\alpha_{0}>0.$  From the definition of $\chi,$ we can estimate $E(t)$ by considering two cases:\medskip \\
	\textbf{Case I:} If $\alpha_{0}E(t)q(t) \leq 2\frac{W_{4}(t)}{\chi(t)} $, then we get
	\begin{equation}
	\label{Egeq}
	E(t)\leq \Big(\frac{2}{\alpha_{0}}\Big)\frac{W_{4}(t)}{\chi(t)q(t)}.
	\end{equation}
	\textbf{Case II:} If $\alpha_{0}E(t)q(t) \geq 2\frac{W_{4}(t)}{\chi(t)} $, then observe that for any $0 \leq s \leq t,$ we have $\alpha_{0}q(s)E(s)>2\frac{W_{4}(s)}{\chi(s)}$. Therefore, we get
	\begin{equation}
	\label{Eleq}
	L_4(s)> 2\frac{W_{4}(s)}{\chi(s)}, \quad 0 \leq s \leq t.
	\end{equation}
	Using Remark (\ref{Remark_1}), $0 < \chi(t) \leq 1$ and the property of $W_{2}$, we have for any $0<\epsilon_{2} \leq 1,$ and $ 0 \leq s \leq t,$
	\begin{equation*}
	\begin{array}{ll}
	\label{ss_11}
	W_{2}(\epsilon_{2}\chi(s) L_{4}(s)-\epsilon_{2}W_{4}(s))& =W_{2}\Big (\epsilon_{2}\chi(s) L_{4}(s)-\frac{\epsilon_{2}\chi(s)W_{4}(s)}{\chi(s)}\Big) \medskip \\
	&\leq \epsilon_{2}\chi (s) W_{2}\Big( L_{4}(s)-\frac{W_{4}(s)}{\chi (s)}\Big )\medskip \\
	&\leq \epsilon_{2}\chi (s)\Big( L_{4}(s)-\frac{W_{4}(s)}{\chi (s)}\Big )W'\Big( L_{4}(s)-\frac{W_{4}(s)}{\chi (s)}\Big ).
	\end{array}
	\end{equation*}
	Using (\ref{Eleq}), we have
	\begin{equation}
	\begin{array}{ll}
	\label{ss_12}
	W_{2}(\epsilon_{2}\chi(t) L_{4}(s)-\epsilon_{2}W_{4}(s)) \leq \epsilon_{2}\chi(s)L_{4}(s)W'(L_{4}(s))-\epsilon_{2}\chi(s)\frac{W_{4}(s)}{\chi(s)}W'\Big(\frac{W_{4}(s)}{\chi(s)}\Big), \quad 0 \leq s \leq t.
	\end{array}
	\end{equation}
	Denote the functional $L_5$ as
	\begin{equation}
	\label{L5}
	L_{5}(t):=\epsilon_{2}\chi (t)L_{4}(t)-\epsilon_{2}W_{4}(t),
	\end{equation}
	where $\epsilon_{2}$ is chosen in such a way that $L_{5}(0) \leq 1$. Using (\ref{w_23}), inequality  (\ref{ss_12}) can be written as,
	\begin{equation}
	\label{ss_13}
	W_{2}(L_{5}(s))\leq \epsilon_{2}\chi (s)W_{2}(L_{4}(s))-\epsilon_{2}\chi (s)W_{2}\Big(\frac{W_4(s)}{\chi (s)}\Big),\quad 0 \leq s \leq t.
	\end{equation}
	Since, $L'_{5}(t)=\epsilon_{2}\chi'(t)L_{4}(t)+\epsilon_{2}\chi (t)L'_{4}(t)-\epsilon_{2}W'_{4}(t),$ using (\ref{ss_10}) we obtain
	\begin{equation*}
	\begin{array}{ll}
	\label{ss_14}
	L'_{5}(t) \leq -C_{1}\xi(t)\epsilon_{2}\chi (t)W_{2}(L_{4}(t))+C_{2}\epsilon_{2}  \xi (t)\chi (t) W^{*}_{3}\Big[\frac{c}{\delta_{1}}q(t)f_{1}^{\frac{1}{1+\epsilon_{0}}}\Big]-\epsilon_{2}W'_{4}(t),
	\end{array}
	\end{equation*}
	applying (\ref{ss_12}) to the above inequality, we get
	\begin{equation}
	\begin{array}{ll}
	\label{ss_14} 
	L'_{5}(t) \leq -C_{1}\xi(t)W_{2}(L_{5}(t))+C_{2}\epsilon_{2}  \xi (t)\chi (t) W^{*}_{3}\Big[\frac{c}{\delta_{1}}q(t)f_{1}^{\frac{1}{1+\epsilon_{0}}}\Big]-C_{1}\epsilon_{2}\xi(t)\chi (t)W_{2}\Big(\frac{W_{4}(t)}{\chi(t)} \Big )-\epsilon_{2}W'_{4}(t).
	\end{array}
	\end{equation}
	From the definition of $W_{1}$ and $W_{4}$ we have $W'_{4}(t)=-C_{1}\xi(t)W_{2}(W_{4}(t)),$ and using (\ref{w_3}) observe that
	\begin{equation*}
	\epsilon_{2}\xi (t)\chi(t)\Bigg(C_{2}W_{3}^{*}\Big[\frac{c}{\delta_{1}}q(t)f_{1}^{\frac{1}{1+\epsilon_{0}}}\Big ]-C_{1}W_{2}\Big (\frac{W_{4}(t)}{\chi(t)}\Big)+C_{1}\frac{W_2(W_{4}(t))}{\chi (t)}\Bigg) \leq 0,
	\end{equation*}
	therefore (\ref{ss_14}) leads to,
	\begin{equation}
	\label{L_5}
	L'_{5}(t) \leq C_{1}\xi(t)W_{2}(L_{5}(t)).
	\end{equation}
	From (\ref{w_23}) and (\ref{L_5}) we obtain
	\begin{equation*}
	\label{k}
	C_{1}\xi(t) \leq(W_{1}(L_{5}(t)))'.
	\end{equation*}
	Integrating the above inequality over $[0,t],$ observe that
	\begin{equation*}
	W_{1}(L_{5}(t)) \geq C_{1}\int_{0}^{t}\xi (s)-W_{1}(L_{5}(0)).
	\end{equation*}
	Since $W_{1}$ is decreasing, $L_{5}(0) \leq 1$ and $W_{1}(1)=0$, we have 
	\begin{equation*}
	L_{5}(t) \leq W_{1}^{-1}\Big ( C_{1}\int_{0}^{t}\xi (s)ds \Big )=W_{4}(t).
	\end{equation*}
	From (\ref{L5}), we get 
	\begin{equation*}
	L_{4}(t) \leq\Big(\frac{1+\epsilon_{2}}{\epsilon_{2}}\Big)\frac{W_{4}(t)}{\chi (t)}.
	\end{equation*}
	Similarly, recalling the definition of the functional $L_4$, we get
	\begin{equation}
	\label{k2}
	L_{3}(t) \leq \Bigg(\frac{1+\epsilon_{2}}{\alpha_{0}\epsilon_{2}}\Bigg)\frac{W_{4}(t)}{\chi (t)q(t)},
	\end{equation}
	since $L_{3} \sim E $, for some $c>0$ we have $E(t) \leq c L_{3}$, then (\ref{k2}) becomes,
	\begin{equation}
	\label{k3}
	E(t)\leq \Bigg(\frac{c(1+\epsilon_{2})}{\alpha_{0} \epsilon_{2}}\Bigg)\frac{W_{4}(t)}{\chi (t)q(t)},
	\end{equation}
	from (\ref{Egeq}) and (\ref{k3}),  we conclude that
	\begin{equation*}
	E(t) \leq \frac{CW_{4}(t)}{\chi (t) q(t)}, \quad {\rm where}\ C= \mathrm{max} \Big (\frac{2}{\alpha_{0}}, \frac{c(1+\epsilon_{2})}{\delta_{0}\epsilon_{2}}\Big).
	\end{equation*}
	Hence the theorem is proved.
\end{proof}
\begin{example}
	Here we present an example which demonstrates the main theorem of this paper. Assume that $\epsilon_{0}\in(0,1)$, $c$ is a positive generic constant and
	\[
	b(t)=c_0(t+1)^{-\frac{1}{p-1}},\ \xi(t)= \frac{c_0^{1-p}}{p-1},\ H(t)=t^{p(1+\epsilon_{0})},
	\]
	where $0<c_0<\frac{2-p}{p-1}$ and $p$ is defined later. Moreover, assume that $u_0$ satisfies
	\[
	\big[1+\|\Delta u_{0}\|^2_2  \big] \sim (1+t)^{\lambda}, \quad \ {\rm for\ }\ \lambda<\frac{2-p}{p-1}.
	\]
	Then recalling the definitions from (\ref{w_23}) and (\ref{w_4}), we obtain:
	\begin{equation*}
	\begin{array}{ll}
	W(t)=ct^{p(1+\epsilon_{0})},\ W_1(t)=c(t^{1-{p(1+\epsilon_{0})}}-1),\ W_2(t)=ct^{p(1+\epsilon_{0})},\ W_3(t)=ct^{\frac{p(1+\epsilon_{0})}{p(1+\epsilon_{0})-1}},\medskip \\
	W_3^*(t)=ct^{p(1+\epsilon_{0})},\ W_4(t)=c(t+1)^{\frac{1}{1-p(1+\epsilon_{0})}}.
	\end{array}
	\end{equation*}
	Note that $q(t)f_1^{\frac{1}{(1+\epsilon_{0})}}(t)\sim (t+1)^{\big(\frac{1}{1+\epsilon_{0}}\big)(\lambda-\frac{1}{p-1})}$.\\ Now, choose $\chi(t)=(t+1)^\gamma,$ where $\gamma<\min\Big(0,-\frac{1}{p(1+\epsilon_{0})-1} + \frac{1+\lambda-\lambda p}{(p-1)(1+\epsilon_{0})}\Big).$ Then it is easy to observe that $\chi(t)$ satisfies \eqref{w_3}. Therefore \eqref{energy_est} implies
	\begin{equation*}
	E(t)\leq
	\left\{  \begin{array}{ll}	
	c (t+1)^{-\frac{1}{1+\epsilon_{0}}\Big(\frac{2-p}{p-1} - \lambda\Big)},&\quad {\rm if\ } 0<\lambda<\frac{2-p}{p-1}\ {\rm and}\ 1<p<2,\medskip \\
	c (t+1)^{-\Big(\frac{2-p+\epsilon_{0}(1-p)}{(1+\epsilon_{0})(p-1+p\epsilon_{0})}\Big)},&\quad {\rm if\ } \lambda\leq 0\ {\rm and}\ 1<p<\frac{3}{2}.
	\end{array}\right. 
	\end{equation*}
	Hence, from the above estimate we conclude that $\lim\limits_{t\rightarrow \infty}E(t)=0.$
\end{example}
\bibliographystyle{plain}
\bibliography{reference_new}
\end{document}